\documentclass{amsart}

\usepackage{stmaryrd}
\usepackage{filecontents}
\usepackage{hyperref}

\newtheorem{theorem}{Theorem}[section]
\newtheorem{lemma}[theorem]{Lemma}

\theoremstyle{definition}

\theoremstyle{remark}

\numberwithin{equation}{section}

\newcommand{\C}{{\mathcal C}}
\newcommand{\D}{{\mathcal D}}

\newcommand{\M}{{\mathcal M}}
\newcommand{\R}{\mathcal{R}}
\newcommand{\NN}{\mathbb{N}}
\newcommand{\RR}{\mathbb{R}}
\newcommand{\ZZ}{\mathbb{Z}}

\newcommand{\abs}[1]{{\left|#1\right|}} 
\newcommand{\norm}[1]{{\left\|#1\right\|}}

\newcommand{\biggnorm}[1]{{\bigg\|#1\bigg\|}}

\begin{document}
\numberwithin{equation}{section}

\title[Absolutely summing operators and atomic decomposition]{Absolutely summing operators and atomic decomposition in bi-parameter Hardy spaces.}
\author{Paul F.X.~M\"uller \and Johanna Penteker}

\address{P.F.X. M\"uller, Institute of Analysis, Johannes Kepler University Linz, Austria, 4040 Linz, Altenberger Strasse 69}
\email{paul.mueller@jku.at}
\address{J. Penteker, Institute of Analysis, Johannes Kepler University Linz, Austria, 4040 Linz, Altenberger Strasse 69}
\email{johanna.penteker@jku.at}
\date{\today}

\maketitle

\subsection*{Abstract}
For $f \in H^p(\delta^2)$, $0<p\leq 2$, with Haar expansion $f=\sum f_{I \times J}h_{I\times J}$ we constructively determine the Pietsch measure of the $2$-summing multiplication operator
	\[\mathcal{M}_f:\ell^{\infty} \rightarrow H^p(\delta^2), \quad (\varphi_{I\times J}) \mapsto \sum \varphi_{I\times J}f_{I \times J}h_{I \times J}.
\]
Our method yields a constructive proof of Pisier's decomposition of $f \in H^p(\delta^2)$
	\[|f|=|x|^{1-\theta}|y|^{\theta}\quad\quad \text{ and }\quad\quad \|x\|_{X_0}^{1-\theta}\|y\|^{\theta}_{H^2(\delta^2)}\leq C\|f\|_{H^p(\delta^2)},
\] where $X_0$ is Pisier's extrapolation lattice associated to $H^p(\delta^2)$ and $H^2(\delta^2)$. 
Our construction of the Pietsch measure for the multiplication operator $\mathcal{M}_f$ involves the Haar coefficients of $f$ and its atomic decomposition. We treated the one-parameter $H^p$-spaces in [{\em Houston Journal Math.}~2015].

\section{Introduction}
Let $Y_0,Y$ be Banach spaces. An operator $T\in L(Y_0,Y)$ is called \textit{2-summing} if there is a constant $C$ such that for every choice of finite sequences $(\varphi_i)$ in $Y_0$, we have
\begin{equation}
\label{eq:summing}
	\bigg(\sum_{i=1}^n{\norm{T\varphi_i}^2}\bigg)^{\frac{1}{2}}\leq C\,\sup{\bigg\{\Big(\sum_{i=1}^n{\abs{\varphi^*(\varphi_i)}}^2\Big)^{\frac{1}{2}}:\,\varphi^* \in B_{Y_0^*}\leq 1\bigg\}}. 
\end{equation}
In the early 70's the concepts of type and cotype were mainly developed by\\ J.~Hoffmann-J{\o}rgensen, S.~ Kwapien, B.~Maurey and G.~Pisier, see \cite{MR0356155,MR0341039,MR0346510,MR0308834,MR0331017,MR0443015,MR0333673}.
A Banach space $Y$ is called of \textit{cotype 2} if there is a constant $C$ such that for all finite sequences $(y_i)$ in $Y$ 
\begin{equation*}
\label{eq:cotype}
	\bigg(\sum_{i=1}^n\norm{y_i}^2\bigg)^{\frac{1}{2}}\leq C \bigg(\int_0^1{\biggnorm{\sum_{i=1}^n{r_i(t)y_i}}_X^2dt}\bigg)^{\frac{1}{2}},
\end{equation*}where $(r_i)_{i \in \NN}$ denotes the independent Rademacher system.
 One famous theorem due to Maurey (\cite{MR0399818, MR0344931}, see also \cite{MR0487403}) combining absolutely summing operators and the concept of cotype states that every bounded operator 
\[T\colon\ell^{\infty}\rightarrow Y\] is $2$-summing, whenever $Y$ is of cotype $2$.
In particular, if
\[\norm{T\varphi}_Y\leq \sup_{i \in \NN}\abs{\varphi_i},\]
then $T$ satisfies \eqref{eq:summing}
and by Pietsch's factorization theorem (cf.~\cite{MR1144277}) there exists a constant $C$ such that
\begin{equation}
\label{eq:pietsch}
\norm{T\varphi}_Y\leq C \left(\int_{\Omega}\abs{\varphi}^2d\mu\right)^{\frac{1}{2}},
\end{equation}	
where $\mu$ is a Borel probability measure on $\Omega=B_{(\ell^{\infty})^*}$, called Pietsch measure.  
Another concept going back to the 70's are Hardy spaces of martingales and their atomic decomposition, cf.~\cite{MR0447953,MR0312166,MR539351,MR602392,MR590628,MR584078}.

In our recent paper \cite{MulPent} we exhibited a connection between these two concepts. 
In the present work we further extend and exploit these newly found connections. We consider operators from $\ell^{\infty}$ into bi-parameter dyadic Hardy spaces $H^p(\delta^2)$ that act as multipliers on the Haar system. By the above, these multiplication operators are $2$-summing and satisfy therefore \eqref{eq:pietsch}.
In our main result (Theorem \ref{th:main1}) we determine explicit formulae for the Pietsch measure of these multiplication operators. We recall that for general absolutely summing operators the existence of a Pietsch measure is given by a Hahn-Banach argument and is therefore not constructive.
Let $\D$ be the set of dyadic intervals.     
Let $(f_{I \times J})_{I \times J\in \D\times \D}$ be a real sequence indexed by the dyadic rectangles $\D\times \D$. The space $H^p(\delta^2)$ consists of all functions
 \begin{equation*}
 \label{eq:func}
 f=\sum_{I \in \D}\sum_{J \in \D}f_{I \times J}h_{I \times J},
 \end{equation*} where $h_{I\times J}=h_I \otimes h_J$, which satisfy
 \begin{equation*}
 \norm{f}_{H^p(\delta^2)}=\left(\int_{[0,1]^2}\Big(\sum_{I\in \D}\sum_{J\in \D}f^2_{I\times J}1_{I\times J}\Big)^{\frac{p}{2}}\,dm\right)^{\frac{1}{p}} <\infty,
 \end{equation*}where $m$ denotes the Lebesgue measure on $[0,1]^2$. Every $f \in H^p(\delta^2)$ defines a multiplication operator of the form
 \begin{equation}
 \label{eq:mult1}
 \begin{split}
 \M_f\colon\ell^{\infty}(\D \times \D) &\rightarrow H^p(\delta^2)\\
  (\varphi_{I \times J}) &\mapsto \sum_{I}\sum_{J}\varphi_{I \times J}f_{I \times J} h_{I \times J}.  
 \end{split}
 \end{equation}
 For $1\leq p \leq 2$ the Hardy spaces $H^p(\delta^2)$ are of cotype $2$ and therefore the multiplication operators $\M_f$ are $2$-summing and have Pietsch measures. 
  In our main theorem (Theorem \ref{th:main1}) we use the atomic decomposition of $f \in H^p(\delta^2)$ to give explicit formulae for these Pietsch measures. In particular, we determine $\omega=(\omega_{I \times J})_{I \times J \in \D \times \D}$ with $\omega_{I \times J}\geq 0$ and $\sum \omega_{I \times J}\leq 1$ such that
 for all $\varphi \in \ell^{\infty}(\D\times \D)$ the following holds
 \begin{equation}
 \label{eq:pietsch2}
 \norm{\M_f(\varphi)}_{H^p(\delta^2)} \leq C \norm{f}_{H^p(\delta^2)}\bigg(\sum_{I \in \D}\sum_{J \in \D}\abs{\varphi_{I\times J}}^2\omega_{I\times J}\bigg)^{\frac{1}{2}}.
 \end{equation}
 The explicit formulae for $\omega$ are given by equation \eqref{eq:omega} in Section \ref{sec:mainth}. 
 Multiplication operators such as given in \eqref{eq:mult1} played an important role in the development of Banach space theory. See for instance the proof by Lindenstrauss and Pe{\l}czy{\'n}ski on the uniqueness of the unconditional basis in $\ell^1$ (\cite{MR0500056,MR0231188}).

 Bi-parameter Hardy spaces $H^p(\delta^2)$ may be regarded as vector-valued Hardy spaces $H^p_X$, where $X=H^p$. In \cite[Theorem 3.3, (3.19)]{MulPent} we obtained partially constructive formulae for the Pietsch measures of Haar multipliers on $\ell^\infty$ into the vector-valued Hardy spaces $H^p_X$. In the scalar-valued case, i.e.~$X=\RR$, we obtained fully constructive formulae for the Pietsch measures of the multiplication operators, see \cite[Theorem 3.1.]{MulPent}.
With this in mind, our present theorem (Theorem \ref{th:main1}) gives fully constructive results for a special class of vector-valued Hardy spaces and simultaneously we extend in a non-trivial way the scalar-valued one-parameter case to the bi-parameter case.   
\smallskip 

\noindent
\textit{Application.} The Banach spaces $H^p(\delta^2)$ form Banach lattices whose lattice structure is induced by their unconditional basis $(h_{I \times J})$ and they are related through Calder{\'o}n's product formula 
\begin{equation}
\label{eq:calderon}
H^p(\delta^2)=\left(H^1(\delta^2)\right)^{1-\theta}\left(H^2(\delta^2)\right)^{\theta}, \quad 0<\theta<1,\,\, \frac{1}{p}=1-\theta+\frac{\theta}{2}.
\end{equation}
This follows by combining the one-parameter identities (cf.~\cite[Theorem 8.2.]{MR1070037}) with Calder{\'o}n's theorem (\cite[Paragraph 13.6]{MR0167830}).
Therefore, Pisier's extrapolation statement (\cite[Theorem 2.10]{MR557371}) can be adapted to the family of $H^p(\delta^2)$ spaces and reads in this setting as follows
\begin{equation*}
\label{eq:pisier}
H^p(\delta^2)=(X_0)^{1-\theta}(H^2(\delta^2))^{\theta}, \quad \theta=2-\frac{2}{p}.
\end{equation*}
Here $X_0$ is the Banach lattice of all elements $x=\sum_{I}\sum_Jx_{I \times J}h_{I \times J}$ for which
\begin{equation}
\label{eq:X0norm}
\norm{x}_{X_0}=\sup{\left\{\biggnorm{\sum_I\sum_J\abs{x_{I \times J}}^{1-\theta}\abs{y_{I \times J}}^{\theta}h_{I \times J}}_{H^p(\delta^2)}\right\}}<\infty,
\end{equation}
where the supremum is taken over all $y=\sum_{I,J}y_{I \times J}h_{I\times J}$ with $\norm{y}_{H^2(\delta^2)}\leq 1$.
Specifically, \eqref{eq:pisier} asserts that given $f \in H^p(\delta^2)$ there is $x \in X_0$ and $y \in H^2(\delta^2)$ such that 
\begin{equation}
\label{eq:pisier2}
\abs{f}=\abs{x}^{1-\theta}\abs{y}^{\theta} \quad \text{and} \quad \norm{x}_{X_0}^{1-\theta}\norm{y}_{H^2(\delta^2)}^{\theta}\leq C\norm{f}_{H^p(\delta^2)}.
\end{equation}
Pisier shows in his proof that the weight $\omega=(\omega_{I \times J})$ given by equation \eqref{eq:pietsch2} yields factors for $f$. Hence, our explicit formulae for $\omega=(\omega_{I \times J})$ determined in Theorem \ref{th:main1} allow us to give constructive factors of $f$ satisfying \eqref{eq:pisier2}.

\section{Preliminaries}
\label{sec:prelim}
\subsection{Bi-parameter Hardy spaces $H^p(\delta^2)$}
The \textit{dyadic intervals} $\D$ on the unit interval are given by 
\begin{equation*}
\D=\left\{\left[2^{-n}(k-1),\,2^{-n}k\right[:\, n,k \in \NN_0,\, 0\leq k < 2^n\right\}
\end{equation*}
and the \textit{dyadic rectangles $\R$} on the unit square are given by $\R=\D\times \D.$
Let $\C \subseteq \R$ be a collection of dyadic rectangles. Then we denote by $\C^*$ the \textit{pointset} covered by the union of all dyadic rectangles in the collection $\C$. 
The space $\ell^{\infty}(\R)$ is the space of all sequences $\varphi=(\varphi_{IJ})_{I \times J \in \R}$, indexed by the dyadic rectangles, with $\norm{\varphi}_{\infty}=\sup_{I\times J\in \R}\abs{\varphi_{IJ}}<\infty.$
For every $I \in \D$ we define the $L^{\infty}$- normalised \textit{Haar function} $h_I$ to be $+1$ on the left half of $I$, $-1$ on the right half of $I$ and zero on $[0,1]\setminus I$.
The an-isotropic \textit{2D Haar system} $(h_{I\times J})_{I \times J \in \R}$ indexed by the dyadic rectangles is defined as follows 
\[h_{I \times J}(s,t):=h_I(s)h_J(t), \quad I,J \in \D, \, (s,t)\in [0,1]^2.\]

Let $(f_{IJ})_{I\times J \in \R}$ be a real sequence and $f=(f_{IJ})_{I\times J \in \R}$ the real vector indexed by the dyadic rectangles. The \textit{square function of $f$} is defined as follows
\begin{equation*}
S(f)(s,t)=\bigg(\sum_{I\times J \in \R}f_{IJ}^21_{I\times J}(s,t)\bigg)^{\frac{1}{2}},\quad (s,t) \in [0,1]^2.
\end{equation*}
The \textit{bi-parameter dyadic Hardy space $H^p(\delta^2)$}, $0<p\leq 2$, consists of vectors $f=(f_{IJ})_{I\times J \in \R}$ for which
\begin{equation*}
\norm{f}_{H^p(\delta^2)}=\left(\int_{[0,1]^2}S^p(f)(s,t)\, dm(s,t)\right)^{\frac{1}{p}} <\infty,
\end{equation*}where $m$ is the Lebesgue measure on $[0,1]^2$. 
Systematically we use the notation $\norm{f}_2=\norm{f}_{H^2(\delta^2)}$.
For convenience we identify $f=(f_{IJ})_{I \times J \in \R}\in H^p(\delta^2)$ with its formal Haar series 
\begin{equation}
\label{eq:haarseries}
f=\sum_{I \times J \in \R}f_{IJ}h_{I \times J}.
\end{equation}

\subsection{Atomic decomposition}
\label{sec:atdec}
Let $0<p\leq 2$ and $f \in H^p(\delta^2)$ with Haar expansion \eqref{eq:haarseries}.
For every $n\in \ZZ$ we define the set
\[F_n=\left\{(s,t) \in [0,1]^2:\, S(f)(s,t)>2^n\right\}\]
and the collection of dyadic rectangles
\[\R_n=\left\{I\times J \in \R:\,\abs{I\times J \cap F_n}>\frac{\abs{I\times J}}{2},\, \abs{I\times J \cap F_{n+1}}\leq \frac{\abs{I\times J}}{2}\right\}.\]
Then  $f=\sum_{n \in \ZZ}f_n$, where 
\[f_n=\sum_{I \times J \in \R_n}f_{IJ}h_{I\times J}\]
and the following inequalities hold
\begin{align}
&\label{eq:atdec}
			\norm{f}^p_{H^p(\delta^2)} \leq \sum_{n \in \mathbb{Z}}{\norm{f_n}^p_{H^p(\delta^2)}}\leq \sum_{n \in \mathbb{Z}}{\abs{\mathcal{R}_n^*}^{1-\frac{p}{2}}\norm{f_n}^p_{2}}\leq A_p \norm{f}^p_{H^p(\delta^2)}.
\end{align}
The family $(f_n, \mathcal{R}_n)_{n \in \ZZ}$ is called the \textit{atomic decomposition} of $u \in H^p(\delta^2)$. 
This decomposition originates from \cite{MR0312166,MR539351,MR590628,MR584078}.

Note that the right-hand side inequality in \eqref{eq:atdec} results from the following 
\begin{equation}
\label{eq:l2at}
\begin{split}
\norm{f_n}_{2}^2&=\int_{[0,1]^2} S^2(f_n)\,dm \leq 2\int_{[0,1]^2} S^2(f_n)1_{F_{n+1}^c} \, dm\leq 2\cdot 2^{2(n+1)}\abs{\R_n^*}\\
&\leq 8\cdot 2^{2n} \Big|\Big\{M_S(1_{F_n})>\frac{1}{2}\Big\}\Big|\leq  C \,2^{2n}\abs{F_n}.
\end{split}
\end{equation}
Here $M_S$ is the strong maximal operator (cf.~\cite{Jessen1935},\cite{MR1315539}) in $[0,1]^2$  given by
\[M_S(1_{F_n})(s,t)=\sup_{R \ni (s,t)}\frac{1}{\abs{R}}\int_R 1_{F_n}\,dm,\]
where the supremum is taken over all rectangles $R$ in $[0,1]^2$ with side length parallel to the axes. 
Boundedness estimates for the strong maximal operator (cf.~\cite{Jessen1935}) give rise to bi-parameter Fefferman-Stein strong maximal operator estimates (cf.~\cite[Theorem 1.]{MR0284802}). We exploit these Fefferman-Stein inequalities in the following form. 
\begin{lemma}
\label{le:bownik}
Fix $\varepsilon >0$. Suppose that for each $I \times J \in \R$ the subset $E_{I \times J} \subseteq I \times J$ is a measurable set with $\frac{\abs{E_{I\times J}}}{\abs{I\times J}}>\varepsilon$. Then for any $f \in H^p(\delta^2)$, $0<p<\infty$, with Haar expansion 
$f=\sum_{I \times J \in \R}f_{IJ}h_{I\times J}$,
the following holds
\begin{equation*}
\norm{f}_{H^p(\delta^2)}\leq C_p(\varepsilon)\, \biggnorm{\Big(\sum_{I \times J \in \R}\abs{f_{IJ}}^21_{E_{I\times J}}\Big)^{\frac{1}{2}}}_{L^p}.
\end{equation*}
\end{lemma}
Frazier and Jawerth (\cite[Theorem 2.7.]{MR1070037}) give a proof for the one-parameter version of this lemma. Their proof can be adapted to the setting above.

\subsection{Modified H\"older inequality}
See \cite[p.~61 (65.)]{MR0046395}.
Let $(\Omega, \Sigma, \mu)$ be a measure space and $r>1$ or $r<0$. Then for all measurable functions $f,g$ on $\Omega$
\begin{equation}
\label{eq:modholder}
\int_{\Omega} f^rg^{1-r}d\mu \geq \left(\int_{\Omega}f d\mu\right)^r \left(\int_{\Omega}g \,d\mu\right)^{1-r}.
\end{equation} 
\medskip

\section{The main Theorem}\label{sec:mainth}
Let $0<p\leq 2$.
Every $f \in H^p(\delta^2)$ defines a multiplication operator of the form
\begin{equation*}
\label{eq:mult2}
\M_f\colon \ell^{\infty}(\R)\rightarrow H^p(\delta^2),\,\, (\varphi_{IJ})\mapsto \sum_{I \times J \in \R} \varphi_{IJ}f_{IJ}h_{I \times J}
\end{equation*}
and clearly we have
\begin{equation*}
\norm{\M_f\colon\ell^{\infty}(\R)\rightarrow H^p(\delta^2)}\leq \norm{f}_{H^p(\delta^2)}.
\end{equation*} 
  Banach space theory as described in the introduction guarantees that these multiplication operators are $2$-summing and satisfy \eqref{eq:pietsch2}.
  In Theorem \ref{th:main1} we determine explicit formulae for the weights $\omega=(\omega_{I \times J})$ given in \eqref{eq:pietsch2}.
     Every multiplication operator $\M_f$ is induced by a function $f \in H^p(\delta^2)$. These functions admit an atomic decomposition $(f_n,\R_n)_{n \in \NN}$ satisfying the equations in \eqref{eq:atdec}. This is the input for our construction and the output is equation \eqref{eq:omega} determining $\omega$ explicitly. 

\begin{theorem}
\label{th:main1}
 Let $0<p\leq 2$ and $f \in H^p(\delta^2)$ with Haar expansion
	\[f=\sum_{I\times J \in \R}{f_{IJ}h_{I\times J}}
\]and atomic decomposition $(f_n, \R_n)_{n \in \ZZ}$. Then the sequence $\left(\omega_{IJ}\right)_{I\times J \in \R}$, defined by
\begin{equation}
\label{eq:omega}
\omega_{IJ}=\frac{1}{A_p\norm{f}^{p}_{H^p(\delta^2)} }\frac{\abs{\R_n^*}^{1-\frac{p}{2}}f_{IJ}^2\abs{I}\abs{J}}{\norm{f_n}^{2-p}_{2}}, \hspace{0.5 cm} I \times J \in \R_n,
\end{equation}satisfies
\begin{equation*}
	\sum_{I\times J \in \R}{\omega_{IJ}}\leq 1
\end{equation*}and for each $\varphi \in \ell^{\infty}(\R)$ the following inequality holds
\begin{equation*}
\label{eq:mainth}
	\norm{\M_f(\varphi)}_{H^p(\delta^2)}\leq A_p \|f\|_{H^p(\delta^2)}\Big(\sum_{I\times J \in \R}{\abs{\varphi_{IJ}}^2 \omega_{IJ}}\Big)^{\frac{1}{2}}.
\end{equation*}
\end{theorem}
\begin{proof}
Note that the left-hand side inequality of (\ref{eq:atdec}) depends only on the fact that $(\R_n)_{n\in \ZZ}$ forms a partition of $\R$. Hence, for all $\varphi=(\varphi_{IJ}) \in\ell^{\infty}(\R)$ the following estimate holds
\begin{equation*}
	\begin{split}
	\biggnorm{\sum_{I\times J \in \R}{\varphi_{IJ} f_{IJ} h_{I\times J}}}^p_{H^p(\delta^2)} &=\biggnorm{\sum_{n \in\ZZ}\sum_{I \times J \in \R_n}{\varphi_{IJ} f_{IJ} h_{I \times J}}}^p_{H^p(\delta^2)}\\
	& \leq \sum_{n \in\ZZ}{\biggnorm{\sum_{I \times J \in \R_n}{\varphi_{IJ} f_{IJ} h_{I \times J}}}^p_{2}\abs{\R_n^*}^{1-\frac{p}{2}}}\\
	&= \sum_{n \in\ZZ}{\biggnorm{\sum_{I \times J \in \R_n}{\varphi_{IJ} \frac{f_{IJ}}{\norm{f_n}_{2}} h_{I \times J}}}^p_{2}\norm{f_n}^p_{2}\abs{\R_n^*}^{1-\frac{p}{2}}}.
	\end{split}
\end{equation*}
With
\begin{align*}
\biggnorm{\sum_{I \times J \in \R_n}{\varphi_{IJ} \frac{f_{IJ}}{\norm{f_n}_{2}} h_{I \times J}}}^p_{2} &= \Big(\sum_{I \times J \in \R_n}{\varphi_{IJ}^2\frac{f_{IJ}^2}{\norm{f_n}_{2}^2}\abs{I}\abs{J}}\Big)^{\frac{p}{2}}
\end{align*}
and H\"older's inequality we get
\begin{equation*}
\begin{split}
&\biggnorm{\sum_{I \times J \in \R}{\varphi_{IJ} f_{IJ} h_{I \times J}}}^p_{H^p(\delta^2)} \leq \sum_{n \in\ZZ}{\Big(\sum_{I \times J \in \R_n}{\varphi_{IJ}^2\frac{f_{IJ}^2}{\norm{f_n}_{2}^2}\abs{I}\abs{J}}\Big)^{\frac{p}{2}}\norm{f_n}^p_{2}\abs{\R_n^*}^{1-\frac{p}{2}}}\\
&\qquad\quad\,\,\,=\sum_{n \in\ZZ}{\Big(\sum_{I \times J \in \R_n}{\varphi_{IJ}^2\frac{f_{IJ}^2}{\norm{f_n}_{2}^2}\abs{I}\abs{J}\norm{f_n}^p_{2}\abs{\R_n^*}^{1-\frac{p}{2}}}\Big)^{\frac{p}{2}}\Big(\norm{f_n}^p_{2}\abs{\R_n^*}^{1-\frac{p}{2}}\Big)^{1-\frac{p}{2}}}\\
&\qquad\quad\,\,\,\leq \Big(\sum_{n \in\ZZ}\sum_{I \times J \in \R_n}{\varphi_{IJ}^2\frac{f_{IJ}^2}{\norm{f_n}^{2-p}_{2}}\abs{I}\abs{J}\abs{\R_n^*}^{1-\frac{p}{2}}}\Big)^{\frac{p}{2}}\Big(\sum_{n \in\ZZ}\norm{f_n}^p_{2}\abs{\R_n^*}^{1-\frac{p}{2}}\Big)^{1-\frac{p}{2}}.
\end{split}
\end{equation*}
\begin{align*}
&\biggnorm{\sum_{I \times J \in \R}{\varphi_{IJ} f_{IJ} h_{I \times J}}}_{H^p(\delta^2)}^p\\
&\qquad\qquad\leq A_p^{1-\frac{p}{2}}\norm{f}^{p(1-\frac{p}{2})}_{H^p(\delta^2)}\Big(\sum_{n\in \ZZ}\sum_{I \times J \in \R_n}{\varphi_{IJ}^2\frac{f_{IJ}^2}{\norm{f_n}^{2-p}_{2}}\abs{I}\abs{J}\abs{\R_n^*}^{1-\frac{p}{2}}}\Big)^{\frac{p}{2}}\\
&\qquad\qquad= A_p\norm{f}^p_{H^p(\delta^2)}\Big(\sum_{n \in\ZZ}\sum_{I \times J \in \R_n}{\varphi_{IJ}^2\frac{f_{IJ}^2}{\norm{f_n}^{2-p}_{2}\norm{f}^{p}_{H^p(\delta^2)}A_p}\abs{I}\abs{J}\abs{\R_n^*}^{1-\frac{p}{2}}}\Big)^{\frac{p}{2}}.
\end{align*}
Recall that
\begin{equation}
\label{eq:pr3}
\norm{f_n}_{2}^2=\sum_{I \times J \in \R_n}{f_{IJ}^2\abs{I}\abs{J}}.	
\end{equation}
By the right-hand side inequality in equation (\ref{eq:atdec}) and by equation (\ref{eq:pr3}) we obtain for the sequence $\left(\omega_{IJ}\right)_{I \times J \in \R}$, defined by
	\[\omega_{IJ}=\frac{1}{A_p\norm{f}^{p}_{H^p(\delta^2)} }\frac{\abs{\R_n^*}^{1-\frac{p}{2}}f_{IJ}^2\abs{I}\abs{J}}{\norm{f_n}^{2-p}_{2}}, \hspace{0.5 cm} I \times J \in \R_n,
\]the following estimate
\begin{align*}
\sum_{I \times J \in \R}{\omega_{IJ}}&=\frac{1}{A_p\norm{f}^p_{H^p(\delta^2)}}\sum_{n \in\ZZ}\sum_{I \times J \in \R_n}{\frac{\abs{\R_n^*}^{1-\frac{p}{2}}f_{IJ}^2\abs{I}\abs{J}}{\norm{f_n}^{2-p}_{2}}}\\
&=\frac{1}{A_p\norm{f}^p_{H^p(\delta^2)}}\sum_{n\in\ZZ}{\abs{\R_n^*}^{1-\frac{p}{2}}\norm{f_n}^p_{2}}\,\leq 1.
\end{align*}
\end{proof}

\section{Another application of the atomic decomposition}
Pisier's extrapolation lattice $X_0$ defined in \eqref{eq:X0norm} is known to coincide with $H^1(\delta^2)$. This follows by a specialisation of a general theorem of Cwikel and Nilsson (see \cite{MR1996919}).
Their extrapolation method is applicable, since $H^p(\delta^2)$ spaces are related through Calderon's product formula (cf.~equation \eqref{eq:calderon}).
The space $X_0$ is of particular importance to our work in this paper. Hence, we take the opportunity to complement the work of \cite{MR1996919,MR1070037} with a direct argument based on the atomic decomposition of $H^p(\delta^2)$. We build our strategy by exploiting the formulae used by \cite{MR2157745, MR2183484, zbMATH06162617} for similar purposes. In particular, we refer to Bownik's paper \cite{zbMATH06162617} for the formula \eqref{eq:gfunc} and the idea of using Lemma \ref{le:bownik} in the proof of the following theorem. 
\begin{theorem}
\label{th:int}
Let $f \in H^p(\delta^2)$, $0<p\leq 2$, with Haar expansion $f=\sum f_{IJ} h_{I \times J}$. Then for $0<\theta<1$ and $q$ given by 
\begin{equation}
\label{eq:intq}
\frac{1}{q}=\frac{1-\theta}{p}+\frac{\theta}{2},
\end{equation}
the following holds:
\begin{equation*}
c_p \norm{f}_{H^p(\delta^2)}^{1-\theta}\leq \sup{\left\{\biggnorm{\sum_{I\times J \in \R}\abs{f_{IJ}}^{1-\theta}\abs{g_{IJ}}^{\theta}h_{I \times J}}_{H^q(\delta^2)}\right\}}\leq \norm{f}_{H^p(\delta^2)}^{1-\theta},
\end{equation*}where the supremum is taken over all functions $g=\sum g_{IJ}h_{I\times J}$ with $\norm{g}_{2}\leq 1$. 
\end{theorem}

\begin{proof}
We start with the proof of the right-hand side inequality. Let 
\[h=\sum_{I \times J \in \R}\abs{f_{IJ}}^{1-\theta}\abs{g_{IJ}}^{\theta}h_{I\times J}.\]
Then, by applying H\"older's inequality for sequence spaces with $1-\theta+\theta=1$, we obtain the following inequality for the square functions
\begin{align*}
S^q(h)\leq S^{q(1-\theta)}(f)\,S^{q\theta}(g). 
\end{align*}
Integrating over $[0,1]^2$ and applying H\"older's inequality with $\frac{q(1-\theta)}{p}+\frac{q\theta}{2}=1$ yields
\begin{equation*}
\norm{h}_{H^q(\delta^2)}\leq \norm{f}_{H^p(\delta^2)}^{1-\theta}\,\,\norm{g}_{2}^{\theta}.
\end{equation*}

For the left-hand side inequality we show that for every $f \in H^p(\delta^2)$ there exists a function $g \in H^2(\delta^2)$ such that $\norm{g}_{2}^2\leq c_p\norm{f}^p_{H^p(\delta^2)}$ and 
\begin{align}
\label{eq:bow0}
 \biggnorm{\sum_{I\times J \in \R}\abs{f_{IJ}}^{1-\theta}\abs{g_{IJ}}^{\theta}h_{I \times J}}_{H^q(\delta^2)}^q&\geq C_p\, \norm{f}^p_{H^p(\delta^2)}.
\end{align}
Let $(f_n,\R_n)_{n \in \ZZ}$ be the atomic decomposition of $f \in H^p(\delta^2)$. 
Let $g=\sum g_{IJ}h_{I \times J}$, where
\begin{equation}
\label{eq:gfunc}
\abs{g_{IJ}}=2^{-\frac{n}{2}(2-p)}\abs{f_{IJ}}, \quad I \times J \in \R_n. 
\end{equation} Then, by equation \eqref{eq:l2at}, we have
\begin{equation}
\label{eq:lhs}
\begin{split}
\norm{g}^2_{2}&=\sum_{n \in \ZZ}2^{-n(2-p)}\norm{f_n}_2^2\leq C\sum_{n \in \ZZ}2^{-n(2-p)}2^{2n}\abs{F_n}=C\sum_{n \in \ZZ}2^{np}\abs{F_n}\\
&\leq c_p\norm{f}_{H^p(\delta^2)}^p.
\end{split}
\end{equation}

To prove equation \eqref{eq:bow0} we use Lemma \ref{le:bownik} with sets $E_{I \times J}=I\times J \cap F_n$, for $I \times J \in \R_n$ and obtain
\begin{equation}
\label{eq:bow3}
\begin{split}
\norm{f}_{H^p(\delta^2)}^p&=\left(\int_{[0,1]^2}S^p(f)\,dm\right)^{\frac{q}{p}}\left(\int_{[0,1]^2} S^p(f)\,dm\right)^{1-\frac{q}{p}}\\
&\leq C_p^q\left(\int_{[0,1]^2}\Big(\sum_{I\times J}\abs{f_{IJ}}^21_{E_{I\times J}}\Big)^{\frac{p}{2}}\,dm\right)^{\frac{q}{p}}\left(\int_{[0,1]^2}S^p(f)\,dm\right)^{1-\frac{q}{p}}.
\end{split}
\end{equation}
Let $h=\left(\sum_{I\times J \in \R}\abs{f_{IJ}}^21_{E_{I\times J}}\right)^{\frac{1}{2}}$.
Note that by the modified H\"older inequality (cf.~equation \eqref{eq:modholder}) we have
\begin{equation}
\label{eq:revholder}
\left(\int_{[0,1]^2} h^p\,dm\right)^{\frac{q}{p}}\left(\int_{[0,1]^2} S^p(f)\,dm\right)^{1-\frac{q}{p}}\leq \int_{[0,1]^2} h^qS^{p-q}(f)\,dm.
\end{equation}
Combining equation \eqref{eq:bow3} and \eqref{eq:revholder} yields
\begin{equation}
\label{eq:bow2}
\begin{split}
\norm{f}^p_{H^p(\delta^2)}&\leq C_p^q\int_{[0,1]^2}\Big(\sum_{I\times J \in \R}\abs{f_{IJ}}^21_{E_{I\times J}}\Big)^{\frac{q}{2}}S^{p-q}(f)\, dm\\
&=C_p^q\int_{[0,1]^2}\Big(\sum_{n\in \ZZ}\sum_{I\times J \in \R_n}\abs{f_{IJ}}^21_{I\times J}1_{F_n}\Big)^{\frac{q}{2}}S^{p-q}(f)\, dm.
\end{split}
\end{equation}
We know that $S(f)1_{F_n}>2^{n}1_{F_n}$. Since $q>p$, it follows that  
\begin{equation}
\label{eq:S2}
S(f)^{p-q}1_{F_n}<2^{-n(q-p)}1_{F_n}.
\end{equation}
Equation \eqref{eq:intq} gives the identity $q-p=\frac{q\theta}{2}(2-p)$. Hence, putting equation \eqref{eq:S2} into equation \eqref{eq:bow2} yields
\begin{equation}
\label{eq:rhs}
\begin{split}
\norm{f}_{H^p(\delta^2)}^p&\leq C_p^q\int_{[0,1]^2}\Big(\sum_{n\in \ZZ}2^{-n\theta(2-p)}\sum_{I\times J \in \R_n}\abs{f_{IJ}}^21_{I\times J}1_{F_n}\Big)^{\frac{q}{2}}\, dm\\
&\leq C_p^q\int_{[0,1]^2}\Big(\sum_{n\in \ZZ}2^{-n\theta(2-p)}\sum_{I\times J \in \R_n}\abs{f_{IJ}}^21_{I \times J}\Big)^{\frac{q}{2}}\, dm\\
&=C_p^q\int_{[0,1]^2}\Big(\sum_{I\times J \in \R}\abs{f_{IJ}}^{2(1-\theta)}\abs{g_{IJ}}^{2\theta}1_{I \times J}\Big)^{\frac{q}{2}}\, dm\\
&=C_p^q\biggnorm{\sum_{I \times J \in \R}\abs{f_{IJ}}^{1-\theta}\abs{g_{IJ}}^{\theta}h_{I \times J}}_{H^q(\delta^2)}^q.
\end{split}
\end{equation}

Summarizing equations \eqref{eq:lhs} and \eqref{eq:rhs} yields
\begin{align*}
\norm{f}_{H^p(\delta^2)}^{1-\theta}\norm{g}_{2}^{\theta}\leq c_{p}^{\frac{\theta}{2}}\norm{f}_{H^p(\delta)^2}^{\frac{p}{q}}\leq C_{p} \,\biggnorm{\sum_{I \times J \in \R}\abs{f_{IJ}}^{1-\theta}\abs{g_{IJ}}^{\theta}h_{I \times J}}_{H^q(\delta^2)}.
\end{align*}
\end{proof}

\subsection*{Acknowledgements}
We would like to thank M.~Bownik for helpful discussions during the preparation of this paper. 

This research has been supported by the Austrian Science foundation (FWF) Pr.Nr.P22549, Pr.Nr.P23987 and Pr.Nr.P28352.
\bibliographystyle{alpha}
\bibliography{bib}
\end{document}